\documentclass[11pt, a4paper]{amsart}
\newtheorem{prop}{Proposition}[section]
\newtheorem{lema}[prop]{Lemma}
\newtheorem{teo}[prop]{Theorem}
\newtheorem{corolario}[prop]{Corollary}

\usepackage[usenames]{color}
\usepackage[all]{xy}

\newcommand{\supp}{\mbox{supp}}
\newcommand{\uno}{1\!\!1}

\title[Valuations on star bodies]{Continuity and representation of valuations on star bodies}

\author{Pedro Tradacete}

\address{Mathematics Department\\ Universidad Carlos III de Madrid \\  28911 Legan\'es (Madrid). Spain.}
\email{ptradace@math.uc3m.es }

\thanks{P. Tradacete acknowledges support from Spanish MINECO under grants MTM2016-76808-P and MTM2016-75196-P.}

\author{Ignacio Villanueva}

\address{Departamento de An\'alisis Matem\'atico \\
Facultad de Matem\'aticas \\ Universidad Complutense de Madrid \\
Madrid 28040}

\email{ignaciov@mat.ucm.es}

\thanks{Research of I. Villanueva was partially supported by grants MTM2014-54240-P, funded by MINECO and QUITEMAD+-CM, Reference: S2013/ICE-2801, funded by Comunidad de Madrid}

\begin{document}

\begin{abstract}
It is shown that every  continuous valuation defined on the $n$-dimensional star bodies has an integral representation in terms of the radial function. Our argument is based on the non-trivial fact that continuous valuations are uniformly continuous on bounded sets. We also characterize the  continuous valuations on the $n$-dimensional star bodies that arise as the restriction of a measure on $\mathbb R^n$.
\end{abstract}



\maketitle

\section{Introduction}

A valuation is a function $V$, defined on a given class of sets $\mathcal F$, which satisfies that for every $A,B\in\mathcal F$
$$
V(A\cup B)+V(A\cap B)=V(A)+V(B).
$$
Valuations can be thought of as a certain generalization of the notion of measure, and have become a relevant area of study in convex geometry. For instance, volume, surface area, and Euler characteristic are distinguished examples of valuations (in the appropriate classes of sets). Historically, valuations were an essential tool in M. Dehn's solution to Hilbert's third problem, asking whether an elementary definition for volume of polytopes was possible. 

The celebrated theorem of H. Hadwiger characterizes continuous rotation and translation invariant valuations on convex bodies as linear combinations of the quermassintegrals \cite{Had}. More recently, S. Alesker provided the characterization of those valuations which are only rotation invariant \cite{Alesker}, as well as those which are only translation invariant \cite{Alesker:01}. We refer to \cite{Alesker,Alesker:01,Lu1,Lu2,LuR} for a broad vision on the role of valuations in convex geometry. Recent developments in valuation theory and its connections with other areas of mathematics can also be found in \cite{Alesker:14}.

Valuations on convex bodies belong to the classical Brunn-Minkowski theory. In \cite{Lut1}, E. Lutwak introduced and developed a dual version of Brunn-Minkowski theory: in this context, convex bodies, Minkowski addition and Hausdorff metric are replaced by star bodies, radial addition and  radial metric, respectively. These have played an important role in the solution of the well-known Busemann-Petty problem \cite{Ga1,Ga2,Zh}, and have become a fundamental area of research \cite{HLYZ,Lut2,Lut3}. D. A. Klain initiated in \cite{Klain96}, \cite{Klain97} the study of rotationally invariant valuations on a specific  class of star-shaped sets, namely those whose radial function is $n$-th power integrable.

In this work we characterize radial continuous valuations on $\mathcal S_0^n$,  the star bodies of $\mathbb R^n$ (i.e. star sets with continuous radial function), in terms of an integral representation. 

Our main results is 

\begin{teo}\label{t:integral}
$V:\mathcal S_0^n\longrightarrow\mathbb R$ is a radial continuous valuation if and only if there exist a finite Borel measure $\mu$ on $S^{n-1}$ and a function $K:\mathbb R_+\times S^{n-1}\rightarrow\mathbb R$ such that
\begin{enumerate}
\item[(a)] $K$ satisfyies the strong Carath\'eodory condition (i.e., for each $s\in\mathbb R_+$ the function $K(s,\cdot)$ is Borel measurable, and for $\mu$-almost every $t\in S^{n-1}$ the function $K(\cdot,t)$ is continuous),
\item[(b)] for every $\lambda\in \mathbb R_+$ there is $G_\lambda\in L^1(\mu)$ such that $K(s,t)\leq G_\lambda(t)$ for $s<\lambda$ and $\mu$-almost every $t\in S^{n-1}$,
\end{enumerate}
and for every star set $L$ with bounded Borel radial function $\rho_L$
\begin{equation}\label{eq:integral}
V(L)=\int_{S^{n-1}} K(\rho_L(t),t)d\mu(t).
\end{equation}
\end{teo}

The proof relies heavily on the fact that radial continuous valuations on $\mathcal S_b^n$, the star sets of $\mathbb R^n$,  are uniformly continuous on bounded sets. See Theorem \ref{t:unif} for the precise statement.

This culminates a series of previous works: In \cite{Vi}, the second named author started the study of continuous valuations on star bodies, characterizing positive rotation invariant valuations as those described by certain integral representation. As a continuation of that work, in \cite{TraVi}, the positivity condition was dropped extending the integral representation to general rotation invariant valuations. 

In addition, the general case (that is, non rotationally invariant valuations) was also studied.  In this direction, it was shown in \cite[Theorem 1.1]{TraVi} that every radial continuous valuation on $\mathcal S_0^n$, the $n$-dimensional star bodies, extends uniquely to a valuation on $\mathcal S_b^n$, the bounded Borel star sets of $\mathbb R^n$. Moreover, using this extension, it can be seen that such a valuation admits an integral representation which is at least valid for star sets with simple radial function. More precisely, \cite[Theorem 1.2]{TraVi} already showed  that for a radial continuous valuation $V:\mathcal S_b^n\to\mathbb R$ there exist a Borel measure $\mu$  on $S^{n-1}$ and  a function $K:\mathbb R_+\times S^{n-1}\rightarrow \mathbb R$ such that, for every star body $L$ whose radial function $\rho_L$ is a simple function,  
one has
$$
V(L)=\int_{S^{n-1}} K(\rho_L(t),t)d\mu(t).
$$

\

Having completed the integral characterization of valuations on star bodies, we will apply it in order to classify the valuations arising from measures in $\mathbb R^n$. Note that, if we think of valuations as possible invariants leading to characterizations of properties of star bodies on $\mathbb R^n$, first of all we should be able to distinguish those valuations that are ``just'' measures on $\mathbb R^n$ from those which are ``strict'' valuations in the sense that they are not induced by any measure in $\mathbb R^n$.

The main tool to do this is the notion of {\em variation} of a valuation, which will be introduced in Section \ref{S:measures}. This will allow us to show that a valuation $V$ arises from a measure in $\mathbb R^n$ if and only if
$V=V_1-V_2$, where $V_1, V_2$ are monotonic increasing valuations. The precise result is

\begin{teo}\label{t:variacion}
Let $V:\mathcal S_0^n\longrightarrow \mathbb R$ be a radial continuous valuation.  Then the following are equivalent:

\begin{enumerate}
\item There exists a (signed) countably additive measure $\nu$ defined on the Borel sets of $\mathbb R^n$ such that, for every $L\in \mathcal S_0^n$, $$\nu(L)=V(L).$$

\item $V$ has bounded variation.

\item $V$ is the difference of two monotonic increasing continuous valuations.

\item There exist $K$ and $\mu$ as in Theorem \ref{t:integral} such that, for $\mu$-almost every $t\in S^{n-1}$, $K(\cdot, t)$ is a continuous function of bounded variation.
\end{enumerate}
\end{teo}

Due to its importance, we explicitly state the case of rotationally invariant valuations. 

\begin{corolario}\label{c:variacion}
Let $V:\mathcal S_0^n\longrightarrow \mathbb R$ be a radial continuous rotationally invariant valuation.  Then the following are equivalent: 

\begin{enumerate}
\item There exists a (signed) countably additive measure $\nu$ defined on the Borel sets of $\mathbb R^+$ such that, for every $L\in \mathcal S_0^n$, $$(\nu\otimes m)(L)=V(L),$$
where $m$ is the Lebesgue measure on $S^{n-1}$ and we use the natural identification between $\mathbb R^n$ and $\mathbb R^+\times S^{n-1}$. 

\item $V$ has bounded variation.

\item $V$ is the difference of two monotonic increasing continuous rotationally invariant valuations. 

\item There exists a continuous function of bounded variation $\theta:\mathbb R^+\longrightarrow \mathbb R$ such that, for every $L\in \mathcal S_0^{n}$, 

$$V(L)=\int_{S^{n-1}} \theta(\rho_L(t)) dm(t),$$
where $m$ is the Lebesgue measure on $S^{n-1}$.
\end{enumerate}
\end{corolario}

\subsection{Connections with previous work}

As we mentioned before, to the best of our knowledge there were no previous results characterizing radial continuous valuations on star bodies. At least not with modern notation. However, when writing \cite{TraVi}, we noted that the papers  \cite{CF, DO1, DO2, DO3, FK, FK:69} actually do speak about valuations. These works study {\em (orthogonally) additive functionals} on certain function spaces, but it is not difficult to show that, in our context, these correspond to valuations. This connection is given with full details in \cite{TraVi}.

The papers  \cite{CF, FK, FK:69} study additive functionals (our radial valuations) on $C(K)$ spaces  with increasing level of generality. The main result in those papers is totally comparable to our Theorem  \ref{t:integral} but with a big  difference: They  impose a priori much more restrictive conditions on the additive functional. In particular they demand that it is continuous, bounded on bounded sets and uniformly continuous on bounded sets, whereas we demand only continuity. It is not difficult to show that continuity implies bounded on bounded sets (see \cite[Lemma 3.1]{TraVi}). In contrast, it is quite hard to show that continuity alone implies uniform continuity on bounded sets. We only finish the proof of that fact in this article, using techniques from \cite{DO1, DO2, DO3} and the full power of our previous results in \cite{TraVi}. Using uniform continuity on bounded sets, the derivation of Theorem \ref{t:integral} follows using some ideas from \cite{DO1}.

\section{Preliminaries and notation}\label{s:preliminar}

Let $S^{n-1}$ denote the euclidean unit sphere in $\mathbb R^n$. We will denote $C(S^{n-1})$ and $B(S^{n-1})$ the spaces of continuous, respectively bounded Borel, real-valued functions on $S^{n-1}$. Also,  $C(S^{n-1})_+$, $B(S^{n-1})_+$ denote the cones of positive functions in $C(S^{n-1})$ and $B(S^{n-1})$, respectively.

Given $x \in \mathbb R^n$, let us denote $[0,x]=\{\lambda x:0\leq \lambda\leq 1\}$, the line segment joining the origin with $x$. A set $L\subset \mathbb R^n$ is a {\em star set} if for every $x\in L$, $[0,x]\subset L$. Let $\mathcal S^n$  denote the family of the star sets of $\mathbb R^n$.

Given $L\in \mathcal S^n$, its {\em radial function} $\rho_L:S^{n-1}\to\mathbb R_+$ is given by
$$
\rho_L(t)=  \sup \{c\geq 0 \, : \, ct\in L\}.
$$
A star set $L$ is called a {\em star body} whenever $\rho_L\in C(S^{n-1})_+$. Conversely, given $f\in C(S^{n-1})_+$ there exists a star body $L_f$ such that $f=\rho_{L_f}$.
Let $\mathcal S_0^n$ denote the set of star bodies in $\mathbb R^n$. Note that star bodies are always bounded.

Analogously, a star set $L$ is a {\em bounded Borel star set} if $\rho_L\in B(S^{n-1})_+$. We denote by $\mathcal S_b^n$ the set of bounded Borel star sets in $\mathbb R^n$.

Given two sets $K,L\in \mathcal S^n$, their {\em radial sum} $K\tilde{+}L$ is defined as the star set with radial function satisfying
$$
\rho_{K\tilde{+}L}=\rho_K+\rho_L.
$$
Note that $K\tilde{+}L\in \mathcal S_0^n$ (respectively, $\mathcal S_b^n$) whenever $K,L\in \mathcal S_0^n$ (respectively, $\mathcal S_b^n$).

The dual analog for the Hausdorff metric of convex bodies is the so-called {\em radial metric}, which can be defined by
$$
\delta(K,L)=\inf\{\varepsilon\geq 0 : K\subset L\tilde{+} \varepsilon B_n, L\subset K\tilde{+} \varepsilon B_n\},
$$
where $B_n$ denotes the euclidean unit ball of $\mathbb R^n$. It is easy to check that
$$
\delta(K,L)=\sup_{t\in S^{n-1}}|\rho_K(t)-\rho_L(t)|=\|\rho_K-\rho_L\|_\infty.
$$

\smallskip

A function $V:\mathcal S^n\longrightarrow \mathbb R$ is a {\em valuation} if for any $K,L\in\mathcal S^n$,
$$
V(K\cup L)+V(K\cap L)=V(K)+ V(L).
$$
Note that if $K,L\in\mathcal S_0^n$ (respectively, $\mathcal S_b^n$), then both $K\cup L$ and $K\cap L$ are in $\mathcal S_0^n$ (respectively, $\mathcal S_b^n$). It is easy to see that
$$
\rho_{K\cup L}=\rho_K\vee \rho_L, \hspace{1cm} \rho_{K\cap L}=\rho_K\wedge \rho_L,
$$
where for functions $f_1, f_2\in B(S^{n-1})_+$, we denote
$$
(f_1\vee f_2)(t)=\max \{f_1(t), f_2(t)\},
$$
$$
(f_1\wedge f_2)(t)=\min \{f_1(t), f_2(t)\}.
$$

With this notation, a valuation $V:\mathcal S_0^n\rightarrow \mathbb R$ induces a mapping $\tilde V:C(S^{n-1})_+\rightarrow \mathbb R$
given by
$$
\tilde V(f)=V(L_f),
$$
where $L_f$ is a star body whose radial function satisfies $\rho_{L_f}=f$. If $V$ is continuous with respect to the radial metric, then  $\tilde V$ is continuous with respect to the $\|\cdot\|_\infty$ norm in $C(S^{n-1})_+$ and satisfies
$$
\tilde V(f)+\tilde V(g)=\tilde V(f\vee g)+\tilde V(f\wedge g)
$$
for every $f,g\in C(S^{n-1})_+$. Conversely, every such function $\tilde V$ induces a continuous valuation on $\mathcal S_0^n$. Similarly, a valuation $V:\mathcal S_b^n\rightarrow \mathbb R$ induces a function $\tilde V:B(S^{n-1})_+\rightarrow \mathbb R$ with analogous properties, and vice versa. Throughout the text, both $V$ and $\tilde V$ will be refered to as valuations (in the corresponding framework).

The following result from \cite[Theorem 1.1.]{TraVi} will be key for our purposes:

\begin{teo}\label{t:extensiontoborel}
If $V:\mathcal S_0^n\longrightarrow \mathbb R$ is a radial continuous valuation, then there exists a unique radial continuous valuation on $\mathcal S_b^n$ extending $V$.
\end{teo}

Given  a function $f:S^{n-1} \longrightarrow \mathbb R$, we denote the support of $f$ by $$supp(f)=\overline{\{t\in S^{n-1} : f(t)\not = 0\}},$$ and for any set $G\subset S^{n-1}$, we will write $f\prec G$ if $\supp(f)\subset G$.

Given a valuation $V$, for each $\lambda>0$, and every Borel set $A\subset S^{n-1}$ we define
\begin{equation}\label{def:mu}
\mu_\lambda(A)=\inf\Big\{ \sup \{\tilde{V}(f): \, f\prec G, \, \|f\|_\infty\leq \lambda\}:\,  A\subset G \mbox{ open}\Big\}.
\end{equation}
It is shown in \cite{Vi} that $\mu_\lambda$ defines a finite regular Borel measure on $\Sigma_n$ (the $\sigma$-algebra of Borel subsets of $S^{n-1}$) which controls the valuation $V$ (see also \cite[Observation 5.1]{TraVi}). In particular, we can consider the countably additive measure
$$
\mu=\sum_{\lambda=1}^\infty \frac{ \mu_\lambda}{2^{\lambda}\|\mu_\lambda\|}.
$$

\begin{prop}\label{p:fi}
Let $\tilde V:B(S^{n-1})_+\longrightarrow\mathbb R$ be a continuous valuation. There is a continuous function $\Phi:B(S^{n-1})_+\longrightarrow L_1(\mu)$ such that for $f\in B(S^{n-1})_+$ we have
$$
\tilde V(f)=\int_{S^{n-1}} \Phi(f)d\mu.
$$
Moreover, for every $f\in B(S^{n-1})_+$ and $A\in\Sigma_n$, $\Phi(f\chi_A)=\Phi(f)\chi_A$.
\end{prop}

\begin{proof}
Fix $f\in B(S^{n-1})_+$ and consider for $A\in \Sigma_n$
$$
\nu_f(A)=\tilde V(f\chi_A).
$$
Since $V$ is a valuation, it follows that $\nu_f$ defines a finitely additive measure on $\Sigma_n$. Moreover, $\nu_f$ is absolutely continuous with respect to $\mu$, so that $\nu_f$ is actually countably additive. Let $\Phi(f)$ be the Radon-Nikodym derivative of $\nu_f$ with respect to $\mu$. Hence, for every $A\in \Sigma_n$ we have
$$
\tilde V(f\chi_A)=\int_A \Phi(f)d\mu.
$$
Moreover, if $\mu(A\cap B)=0$, then $\nu_{f\chi_A}(B)=0$, which yields that $\Phi(f\chi_A)=\Phi(f)\chi_A$, for every $A\in \Sigma_n$.

In order to see that $\Phi:B(S^{n-1})_+\longrightarrow L_1(\mu)$ is continuous, we will first need the following:

\textbf{Claim:} Given $f\in B(S^{n-1})_+$,  for every $\varepsilon>0$ there exists $\delta>0$ such that for every Borel set $A\in \Sigma_n$ and $g\in B(S^{n-1})_+$, if $\|f-g\|_\infty<\delta$, then
$$
\Big|\int_A\Phi(f)-\Phi(g)d\mu\Big|<\varepsilon.
$$

Indeed, the continuity of $\tilde V$ implies that, given $\varepsilon$, there exists $\delta$ such that, for every $h\in B(S^{n-1})^+$, if $\|f -h\|_\infty<\delta$ then $|\tilde V(f)-\tilde V(h)|<\varepsilon$.

Let $A\subset S^{n-1}$ be a Borel set, and $g\in B(S^{n-1})_+$ with $\|f-g\|_\infty<\delta$. Let $h=g\chi_A+f\chi_{A^c}$, which clearly satisfies $\|f-h\|_\infty<\delta$. It follows that
$$
\Big|\int_A\Phi(f)-\Phi(g)d\mu\Big|=|\tilde V(f\chi_A)-\tilde V(g\chi_A)|=|\tilde V(f)-\tilde V(h)|<\varepsilon.
$$
This proves the claim.

Finally, given $f\in B(S^{n-1})_+$ and $\varepsilon>0$, let $\delta>0$ be as in the claim. Suppose $g\in B(S^{n-1})_+$ satisfies $\|f-g\|_\infty<\delta$. Let
$$
A=\{t\in S^{n-1}:\Phi(f)(t)-\Phi(g)(t)>0\}.
$$
We have that
$$
\|\Phi(f)-\Phi(g)\|_1=\int_A\Phi(f)-\Phi(g)d\mu + \int_{A^c}\Phi(g)-\Phi(f)d\mu<2\varepsilon.
$$
\end{proof}

The previous result can be considered as a factorization property of valuations on $B(S^{n-1})_+$, in the sense that there is a commutative diagram
$$\xymatrix{B(S^{n-1})_+\ar[rr]^{\tilde V}\ar[rd]_\Phi&&\mathbb R\\
&L_1(\mu)\ar[ur]_{i}&}
$$
with $i(f)=\int_{S^{n-1}}fd\mu$, so that all arrows are continuous valuations.

\section{Uniform continuity on bounded sets}
In this section we prove our main technical result Theorem \ref{t:unif}. It states that continuous valuations on $\mathcal S_0^{n}$ are uniformly continuous on bounded sets.

The main ingredients of the proof are our results from \cite{TraVi}, techniques appearing in \cite{DO3}  and the well known Kadec-Pe\'lcynski dichotomy from functional analysis.

Throughout this section,  $V$ will be a continuous valuation on $\mathcal S_0^n$, and  $\tilde V:B(S^{n-1})_+\longrightarrow \mathbb R$ will denote the induced extension given by Theorem \ref{t:extensiontoborel}.

\begin{lema}\label{pointwise continuity}
Let $\lambda>0$ and let $\mu_\lambda$ be the measure associated to $V$, $\lambda$ defined in Equation (\ref{def:mu}). Let also  $(f_n)_{n\in \mathbb N}\subset B(S^{n-1})_+$, with $\|f_n\|_\infty\leq \lambda$ for every $n\in \mathbb N$, and $f\in B(S^{n-1})_+$ be such that $f_n\rightarrow f$ $\mu_\lambda$-almost everywhere. Then $\tilde V(f_n)\rightarrow \tilde V(f)$.
\end{lema}

\begin{proof}
Let $\epsilon>0$. Using Egorov's Theorem, we obtain $A\in\Sigma_n$, with $\mu_\lambda(A)\leq \frac{\epsilon}{2}$, such that $f_n\chi_{A^c}\rightarrow f\chi_{A^c}$ uniformly. Then, using the continuity of $\tilde{V}$, we obtain the existence of $n_0$ such that, for every $n\geq n_0$, $$|\tilde{V}(f_n\chi_{A^c})-\tilde{V}(f\chi_{A^c})|<\frac{\epsilon}{2}.$$

Therefore, for every $n\geq n_0$,

$$|\tilde{V}(f_n)-\tilde{V}(f)|\leq |\tilde{V}(f_n\chi_{A^c})-\tilde{V}(f\chi_{A^c})| + |\tilde{V}(f_n\chi_{A})-\tilde{V}(f\chi_{A})| \leq \epsilon. $$

\end{proof}

We will need the following technical result (see \cite[Lemma 2.1]{DO3}).

\begin{lema}\label{ortogonalidad}
Let $\Phi:B(S^{n-1})_+\longrightarrow L_1(\mu)$ such that $$\Phi(f\chi_A)=\Phi(f)\chi_A$$ for every $f\in B(S^{n-1})_+$ and $A\in \Sigma_n$. It holds that:

\begin{enumerate}

\item For every finite sequences $(f_i)_{i=1}^n$, $(g_i)_{i=1}^n$ in $B(S^{n-1})_+$, there exist  $f,g\in B(S^{n-1})_+$ such that,
\begin{enumerate}
\item[(i)] $|f(t)|\leq \max_i |f_i(t)|$ for every $t\in S^{n-1}$,
\item[(ii)] $|g(t)|\leq \max_i |g_i(t)|$ for every $t\in S^{n-1}$,
\item[(iii)] $|f(t)-g(t)|\leq \max_i |f_i(t)-g_i(t)|$ for every $t\in S^{n-1}$,
\item[(iv)] $\Phi(f)-\Phi(g)=\sup_{1\leq i \leq n} (\Phi(f_i)-\Phi(g_i)).$
\end{enumerate}

\item For every finite sequences $(f_i)_{i=1}^n$, $(g_i)_{i=1}^n$ in $B(S^{n-1})_+$, let $f=\sup_i f_i$ for $\mu$ almost every $t$. There exists $g\in B(S^{n-1})_+$ such that,
\begin{enumerate}
\item[(i)]  $|g(t)|\leq \sup_i |g_i(t)|$ for every $t\in S^{n-1}$,
\item[(ii)] $|f(t)-g(t)|\leq \sup_i |f_i(t)-g_i(t)|$ for every $t\in S^{n-1}$,
\item[(iii)] $\Phi(f)(t)-\Phi(g)(t) \geq \inf_i \Phi(f_i)(t)-\Phi(g_i)(t)$ in $\mu$-almost every $t$.
\end{enumerate}
\end{enumerate}
\end{lema}

\begin{proof} (1). Let
$$
B^j=\{t\in S^{n-1} : \Phi(f_j)(t)-\Phi(g_j)(t)=\sup_{1\leq i \leq n} (\Phi(f_i)-\Phi(g_i)\}.
$$
Let $A^1=B^1$ and, for $2\leq j \leq n$, let $A^j=B^j\setminus \bigcup_{i=1}^{j-1}A^i$. Then, we define
$$
f=\sum_{i=1}^n f_i(t)\chi_{A^i}(t), \quad g=\sum_{i=1}^n g_i(t)\chi_{A^i}(t).
$$
It is easy to check that these fulfil the required conditions.

(2). Let $$B^j=\{t\in S^{n-1} : f_j(t)=\sup_i f_i(t)\}.$$ Let $A^1=B^1$ and, for $2\leq j \leq n$, let $A^j=B^j\setminus \bigcup_{i=1}^{j-1}A^i$. Then, we define
$$
g=\sum_{i=1}^n g_i(t)\chi_{A^i}(t).
$$
It is easy to check that $g$ satisfies the conditions demanded.
\end{proof}

The following is a regularity property related to the lattice structure of $L_1(\mu)$. It has also been used in \cite{DO3} in a different space for similar purposes.

\begin{lema}\label{regularityL1}
Let $\mu$ be a finite positive measure on a measurable space $(\Omega,\Sigma)$. Let $(E_n)_{n\in \mathbb N}$ be a sequence of countable subsets of $L_1(\mu)$. That is, for each $n\in \mathbb N$ let
$E_n:=\{ \varphi_j^n: j\in \mathbb N\}, $ where, for every $j\in \mathbb N$, $\varphi_j^n\in L_1(\mu)$. Take
$$
\varphi^n:=\sup_{j\in \mathbb N} \varphi_j^n,
$$
and assume that
$$
\lim_{n\rightarrow \infty} \varphi^n=\varphi
$$
$\mu$-almost everywhere, for some measurable function $\varphi$. Then, for every $n\in \mathbb N$ there exists $j_n\in \mathbb N$ such that, $\mu$-almost everywhere, we have
$$
\lim_{n\rightarrow \infty} \sup_{1\leq j \leq j_n} \varphi_j^{j_n}=\varphi.
$$
\end{lema}

\begin{proof}
For fixed $n\in \mathbb N$,  we have that
$$
\lim_{m\rightarrow \infty} \sup_{1\leq j \leq m} \varphi_j^n=\varphi^n
$$
holds $\mu$-almost everywhere. Hence, by Egorov's Theorem, there is a set $A_n\in \Sigma$ such that $\mu(A_n^c)<2^{-n}$ and such that the sequence $(\sup_{1\leq j \leq m} \varphi_j^n(t))_m$ converges  to $\varphi^n(t)$ uniformly for $t\in A_n$. Then, there exists $j_n\in \mathbb N$ such that, for every $t\in A_n$,  and for every $m\geq j_n$, $$\left|\sup_{1\leq j \leq m} \varphi_j^n(t)- \varphi^n(t)\right|<\frac{1}{n}.$$

The set
$$
A=\bigcup_{n=1}^\infty \bigcap_{k=n}^\infty A_k
$$
satisfies that
$$
\mu(A^c)=\mu\left(\bigcap_{n=1}^\infty \bigcup_{k=n}^\infty A_k^c\right)=0.
$$

Let now $\psi^n=\sup_{1\leq j \leq j_n} \varphi_j^n$. Let
$$
B:=\{t\in \Omega : \varphi^n(t)\rightarrow \varphi(t)\}.
$$

Clearly $\mu((A\cap B)^c)=0$ and for $t\in A\cap B$ it holds that
$$
\lim_{n\rightarrow \infty} \psi^n(t)=\varphi(t).
$$
\end{proof}

The same proof shows a similar regularity property for $B(\Omega)$, the space of bounded measurable functions:

\begin{lema}\label{regularityB}
Let $\mu$ be a finite positive measure on a measurable space $(\Omega,\Sigma)$. Let $(E_n)_{n\in \mathbb N}$ be a sequence of countable subsets of $B(\Omega)$. That is, for each $n\in \mathbb N$ let $$
E_n:=\{ \varphi_j^n: j\in \mathbb N\},
$$
where $\varphi_j^n\in B(\Omega)$ for every $j\in \mathbb N$. Let
$$
\varphi^n:=\sup_{j\in \mathbb N} \varphi_j^n,
$$
and assume that the limit $$\lim_{n\rightarrow \infty} \varphi^n=\varphi$$ exists pointwise. Then, there exists a set $A\in\Sigma$ such that $\mu(A^c)=0$ and for every $n\in \mathbb N$ there exists $j_n\in \mathbb N$ such that, for $t\in A$
$$
\lim_{n\rightarrow \infty} \sup_{1\leq j \leq j_n} \varphi_j^{j_n}(t)=\varphi(t).
$$
\end{lema}

The next result is based on \cite[Lemma 2.3]{DO3}.

\begin{lema}\label{l:aeconverge}
Let $\Phi:B(S^{n-1})_+\longrightarrow L_1(\mu) $ be a continuous orthogonally additive function. Let $(f_n)_{n\in \mathbb N}, (g_n)_{n\in \mathbb N}$ be two sequences in the unit ball of $B(S^{n-1})_+$ such that $||f_n-g_n\|_\infty\rightarrow 0$. Then $\Phi(f_n)(t)-\Phi(g_n)(t)\rightarrow 0$ in $\mu$-almost every $t$.
\end{lema}

\begin{proof}
Let $(f_n)_{n\in \mathbb N}, (g_n)_{n\in \mathbb N}$ be two sequences as  in the hypothesis. If the result is not true, then there exists $\epsilon>0$ such that at least one of the sets $A^+$, $A^-$ has strictly positive measure, where
$$
A^+=\{t\in S^{n-1} : \limsup_{n\rightarrow \infty} \Phi(f_n)(t)- \Phi(g_n)(t)>\epsilon\},
$$
$$
A^-=\{t\in S^{n-1} : \limsup_{n\rightarrow \infty} \Phi(g_n)(t)- \Phi(f_n)(t)>\epsilon\}.
$$
We assume that $\mu(A^+)>0$, the other case being entirely similar.

We apply Lemma \ref{regularityL1} to the sets $E_n:=\{\Phi(f_k) - \Phi(g_k) :\, k\geq n\}$ and we obtain the existence of a sequence of natural numbers $(k_n)_{n\in \mathbb N}$ such that
$$
\limsup_{n\rightarrow \infty} \Phi(f_n)(t)- \Phi(g_n)(t)=\lim_{n\rightarrow \infty} \left(\sup_{n\leq k \leq k_n} \Phi(f_k)(t)- \Phi(g_k)(t)\right)
$$
for $\mu$-almost every $t$.

Let $A':=\{t\in S^{n-1} :\lim_{n\rightarrow \infty} \left(\sup_{n\leq k \leq k_n}  \Phi(f_k)(t)- \Phi(g_k)(t)\right)>\epsilon\}$. We have $\mu(A')=\mu(A^+)>0$.

Now, note that the set
$$
A'':=\bigcup_{m=1}^\infty \bigcap_{n=m}^\infty \left\{t\in S^{n-1}: \sup_{n\leq k \leq k_n}  \Phi(f_k)(t)- \Phi(g_k)(t)>\epsilon\right\}
$$
satisfies $A'\subset A''$.

Hence, there exists $m\in \mathbb N$ such that the set
$$
B=\bigcap_{n=m}^\infty \left\{t\in S^{n-1}: \sup_{n\leq k \leq k_n}  \Phi(f_k)(t)- \Phi(g_k)(t)>\epsilon\right\}
$$
has strictly positive measure. ``Shifting'' the sequences, we may assume that $m=1$.

Now, for every $n\in \mathbb N$ we apply the first part of Lemma \ref{ortogonalidad} to the sequences $(f_k)_{k=n}^{k_n}$, $(g_k)_{k=n}^{k_n}$ and we obtain functions $f^n, g^n$ in $B(S^{n-1})_+$ with $\|f^n\|_\infty,\|g^n\|_\infty\leq 1$, such that, for every $t\in S^{n-1}$
$$
|f^n(t)-g^n(t)|\leq \sup_{n\leq k \leq k_n}  |f_k(t)-g_k(t)|
$$
and
$$
\Phi(f^n)-\Phi(g^n)=   \sup_{n\leq k \leq k_n} \Phi(f_k)-\Phi(g_k).
$$

Therefore, the sequences $(f^n)_{n\in \mathbb N}, (g^n)_{n\in \mathbb N}$ satisfy that $\|f^n-g^n\|_\infty\rightarrow 0$ and $\Phi(f^n)(t)-\Phi(g^n)(t)>\epsilon$ for every $t\in B$.

We consider now the function $f:S^{n-1}\longrightarrow \mathbb R$ defined by
$$
f(t)=\limsup_{n\rightarrow \infty} f^n(t).
$$
It is easy to see that $f\in B(S^{n-1})$ and $\|f\|_\infty\leq 1.$ For every $n\in \mathbb N$, we apply now Lemma \ref{regularityB} to the sets $E_n=\{f^k:\, k\geq n\}$ and we obtain a sequence of natural numbers $(j_n)_{n\in \mathbb N}$ such that
$$
f(t)=\lim_{n\rightarrow \infty} \left(\sup_{n\leq k \leq j_n} f^k(t)\right)
$$
for $\mu$-almost every $t$.

We will use the notation $\sup_{n\leq k \leq k_n} f^k(t)=\tilde{f}_n(t)$.

For every fixed $n\in \mathbb N$, we apply now the second part of Lemma \ref{ortogonalidad} to the finite sequences $(f^k)_{k=n}^{j_n}$ and $(g^k)_{k=n}^{j_n}$, and we obtain a function $\tilde{g}_n\in B(S^{n-1})_+$ such that

\begin{itemize}
\item $\|\tilde{g}_n\|_\infty\leq 1$,
\item $|\tilde{g}_n(t)-\tilde{f}_n(t)|\leq \underset{n\leq k \leq j_n}{\sup} |f^k(t)-g^k(t)|$,
\item $\Phi(\tilde{f}_n)(t)-\Phi(\tilde{g}_n)(t)\geq \underset{n\leq k \leq j_n}{\inf} \Phi(f^k)(t)-\Phi(g^k)(t)$.

\end{itemize}

Therefore, $\|\tilde{f}_n-\tilde{g}_n\|_\infty\rightarrow 0$.

Since $\tilde{f}_n\rightarrow f$ in $\mu$ almost every $t$, we also have that $\tilde{f}_n\chi_{B}\rightarrow f\chi_{B}$, $\mu$ almost everywhere. By Lemma \ref{pointwise continuity}, this implies that
$$
\int_{B} \Phi(\tilde{f}_n)d\mu=\int_{S^{n-1}} \Phi(\tilde{f}_n\chi_{B})d\mu=V(\tilde{f}_n\chi_{B})\longrightarrow V(f\chi_{B}).
$$

On the other hand, $\tilde{g}_n\chi_{B}\rightarrow f\chi_{B}$ in $\mu$ almost every $t$ (because $\|\tilde{f}_n-\tilde{g}_n\|_\infty\rightarrow 0$). So, we also have
$$
\int_{B}  \Phi(\tilde{g}_n)d\mu=V(\tilde{g}_n\chi_{B})\longrightarrow V(f\chi_{B}).
$$
Therefore, it follows that
$$
\int_{B} \Phi(\tilde{f}_n)- \Phi(\tilde{g}_n))d\mu\longrightarrow 0.
$$
However, this is a contradiction with the fact that, for every $t\in B$, $$\Phi(\tilde{f}_n)(t)-\Phi(\tilde{g}_n)(t)>\epsilon.$$
\end{proof}

To finish our proof we will need some more tools from functional analysis:

Let us recall that a set $\mathcal F\subset L_1(\Omega,\Sigma,\mu)$ is uniformly integrable if for every $\varepsilon>0$ there is $\delta>0$ such that for every $B\in \Sigma$ with $\mu(B)<\delta$
$$
\sup_{f\in\mathcal F}\int_B |f| d\mu<\varepsilon.
$$

The following is folklore:

\begin{lema}\label{l:unif+ae}
Let $(f_n)_{n\in\mathbb N}\subset L_1(\Omega,\Sigma,\mu)$ with $\mu(\Omega)<\infty$. If $f_n\rightarrow 0$ $\mu$-almost everywhere, and the sequence $(f_n)_{n\in\mathbb N}$ is uniformly integrable, then $\|f_n\|_1\rightarrow0$.
\end{lema}

\begin{proof}
Given $\varepsilon>0$, take $\delta>0$ such that for every $A\in\Sigma$ with $\mu(A)<\delta$
$$
\sup_{n\in\mathbb N}\int_A |f_n| d\mu<\frac{\varepsilon}{2}.
$$
Now, using Egoroff's theorem, there exist $B\in \Sigma$ with $\mu(B)<\delta$ and such that $\|f_n \chi_{B^c}\|_\infty\rightarrow0$. Thus, we can take $N\in\mathbb N$ such that for $n\geq N$
$$
\|f_n \chi_{B^c}\|_\infty<\frac{\varepsilon}{2\mu(\Omega)}.
$$
It therefore follows that for $n\geq N$ we get
$$
\|f_n\|_1\leq \|f_n\chi_B\|_1 + \| f_n \chi_{B^c}\|_1 \leq \int_B |f_n| d\mu +  \mu(\Omega)\| f_n \chi_{B^c}\|_\infty<\varepsilon.
$$
\end{proof}

We need a decomposition result, which was first given for $L_p$ spaces (for $1\leq p<\infty$) by M. Kadec and A. Pelczynski in \cite{KP} (see also \cite[Theorem 29]{AO}).

\begin{lema}\label{l:ssp}
Let $(f_n)_{n\in\mathbb N}\subset L_1(\Omega,\Sigma,\mu)$ be such that $\sup_{n\in\mathbb N}\| f_n\|_{L_1}<\infty$. Then there is a subsequence $(f_{n_k})_{k\in\mathbb N}$ and a sequence of pairwise disjoint measurable sets $(A_k)_{k\in\mathbb N}$ such that the sequence $(f_{n_k}\chi_{A_k^c})_{k\in\mathbb N}$ is uniformly integrable.
\end{lema}

Now we are ready prove that continuous valuations are uniformly continuous on bounded sets. This will be a key step for the proof of  Theorem \ref{t:integral}
\begin{teo}\label{t:unif}
If $V:\mathcal S_b^n\rightarrow\mathbb R$ is a radial continuous valuation, then it is uniformly continuous on bounded sets. That is, for every $\lambda>0$ and every $\varepsilon>0$ there is $\delta>0$ such that whenever $f,g\in B(S^{n-1})_+$ with $\|f\|_\infty,\|g\|_\infty\leq \lambda$ and  $\|f-g\|_\infty\leq\delta$, we have
$$
|\tilde V (f)-\tilde V(g)|\leq\varepsilon.
$$
\end{teo}

\begin{proof}
Let $\tilde V:B(S^{n-1})_+\rightarrow \mathbb R$ be the valuation induced by $V$. If the result is not true, then there is $\lambda>0$, $\varepsilon>0$ and $(f_n)_{n\in\mathbb N},(g_n)_{n\in\mathbb N}\subset B(S^{n-1})_+$ such that
\begin{itemize}
\item $\|f_n\|_\infty,\|g_n\|_\infty\leq \lambda$ for every $n\in\mathbb N$,
\item $\|f_n-g_n\|_\infty\rightarrow0$,
\item $|\tilde V (f_n)-\tilde V(g_n)|>\varepsilon$ for every $n\in\mathbb N$.
\end{itemize}
For simplicity we will take $\lambda=1$.

Let $\Phi:B(S^{n-1})_+\rightarrow L_1(\mu)$ the mapping given in Proposition \ref{p:fi}. Since $\Phi(f\chi_A)=\Phi(f)\chi_A$ for every $f\in B(S^{n-1})_+$ and $A\in \Sigma_n$, in particular $\Phi$ is orthogonally additive.

Let $h_n=\Phi(f_n)-\Phi(g_n)$. Since $\tilde V$ is bounded on bounded sets \cite[Lemma 3.1]{TraVi} we have that
$$
\sup_{n\in\mathbb N} \|h_n\|_1\leq \sup_{n\in\mathbb N} \Big(\tilde V(f_n)+\tilde V(g_n)\Big)<\infty.
$$
Therefore, by Lemma \ref{l:ssp} there is a subsequence $(h_{n_k})_{k\in\mathbb N}$ and a pairwise disjoint sequence $(A_k)_{k\in\mathbb N}$ such that $(h_{n_k}\chi_{A_k^c})_{k\in\mathbb N}$ is uniformly integrable.

By Lemma \ref{l:aeconverge} we have that $h_n\rightarrow0$ $\mu$-almost everywhere. In particular, so does the sequence $(h_{n_k}\chi_{A_k^c})_{k\in\mathbb N}$. Hence, by Lemma \ref{l:unif+ae}, we have that
\begin{equation}\label{eq:unif}
\|h_{n_k}\chi_{A_k^c}\|_1\rightarrow0.
\end{equation}

On the other hand, since $(A_k)_{k\in\mathbb N}$ are pairwise disjoint, for every $m\in\mathbb N$ we have
$$
\sum_{k=1}^m \Phi(f_{n_k})\chi_{A_k}=\Phi\Big(\sum_{k=1}^m f_{n_k}\chi_{A_k}\Big).
$$
Now, as $\| \sum_{k=1}^m f_{n_k}\chi_{A_k}\|_\infty\leq1$, and $\tilde V$ is bounded on bounded sets \cite[Lemma 3.1]{TraVi} it follows that for some $K>0$ and every $m\in\mathbb N$
$$
\sum_{k=1}^m \|\Phi(f_{n_k})\chi_{A_k}\|_1=\Big\| \sum_{k=1}^m \Phi(f_{n_k})\chi_{A_k}\Big\|_1=\tilde V\Big(\sum_{k=1}^m f_{n_k}\chi_{A_k}\Big)\leq K.
$$
Hence, necessarily we have that
\begin{equation}\label{eq:fn}
\|\Phi(f_{n_k})\chi_{A_k}\|_1\rightarrow0.
\end{equation}

Similarly, we have that
\begin{equation}\label{eq:gn}
\|\Phi(g_{n_k})\chi_{A_k}\|_1\rightarrow0.
\end{equation}

Therefore, putting together \eqref{eq:unif}, \eqref{eq:fn} and \eqref{eq:gn} we get
$$
\|\Phi(f_{n_k})-\Phi(g_{n_k})\|_1\leq \|h_{n_k}\chi_{A_k^c}\|_1+\|\Phi(f_{n_k})\chi_{A_k}\|_1+\|\Phi(g_{n_k})\chi_{A_k}\|_1\rightarrow0.
$$
Hence,
$$
|\tilde V(f_{n_k})-\tilde V(g_{n_k})|=\Big|\int_{S^{n-1}}\Phi(f_{n_k})-\Phi(g_{n_k})d\mu\Big|\leq \|\Phi(f_{n_k})-\Phi(g_{n_k})\|_1\rightarrow0,
$$
which is a contradiction with the fact that $|\tilde V (f_n)-\tilde V(g_n)|>\varepsilon$ for every $n\in\mathbb N$, so the proof is finished.
\end{proof}

\section{Integral representation}

Once we have proved that continuous valuations are actually uniformly continuous on bounded sets, we can prove Theorem \ref{t:integral}.

\begin{proof} [Proof of Theorem \ref{t:integral}]
Suppose first $K:\mathbb R_+\times S^{n-1}\rightarrow\mathbb R$ satisfies (a) and (b) in the statement of the theorem, and we set
$$
V(L)=\int_{S^{n-1}} K(\rho_L(t),t)d\mu(t).
$$
It is clear that $V$ satisfies
$$
V(L_1\cup L_2)+V(L_1\cap L_2)=V(L_1)+V(L_2).
$$
Concerning continuity, let $(L_n)$ be a sequence of star bodies converging in the radial metric to $L$, that is if we denote $f_n=\rho_{L_n}$ and $f=\rho_L$, then $\|f_n-f\|_\infty \rightarrow0$. In particular, for every $t\in S^{n-1}$ it follows that $f_n(t)\rightarrow f(t)$, and by (a) we have $K(f_n(t),t)\rightarrow K(f(t),t)$ for $\mu$-almost every $t\in S^{n-1}$. Denoting $\lambda=\sup_n \|f_n\|_\infty$, by (b) we have that $K(f_n(t),t)\leq G_\lambda(t)$ for $\mu$-almost every $t\in S^{n-1}$. Hence, by the dominated convergence theorem it follows that
$$
V(L_n)=\int_{S^{n-1}}K(f_n(t),t) d\mu(t)\rightarrow \int_{S^{n-1}}K(f(t),t) d\mu(t)=V(L).
$$
Hence, $V$ defines a continuous valuation.

For the converse implication, given a radial continuous valuation $V:\mathcal S_0^n\longrightarrow\mathbb R$, we consider the induced mapping $\tilde V:B(S^{n-1})_+\rightarrow\mathbb R$ (see section \ref{s:preliminar}) and for $\lambda>0$, the Borel measures $\mu_\lambda$ given in \eqref{def:mu}. Let
$$
\mu=\sum_{\lambda=1}^\infty \frac{ \mu_\lambda}{2^{\lambda}\|\mu_\lambda\|}.
$$

Let $\Phi:B(S^{n-1})_+\rightarrow L_1(\mu)$ be the mapping given in Proposition \ref{p:fi}. Recall that for each $s\in \mathbb R_+$, $\Phi(s\chi_{S^{n-1}})$ is a Radon-Nikodym derivative of the measure $\nu_s$ with respect to $\mu$, where $\nu_s(A)=\tilde V(s\chi_A)$ for each $A\in \Sigma_n$. Let us define for $s\in\mathbb R_+$, $t\in S^{n-1}$, the function
$$
K_0(s,t)=\Phi(s\chi_{S^{n-1}})(t).
$$
Now, by Theorem \ref{t:unif}, we know that $\tilde V$ is uniformly continuous on bounded sets, and an argument like that of \cite[2.1.3.]{DO1} or \cite[Lemma 11]{CF} (see Lemma \ref{l:unifrational} in the Appendix for details) yields a Borel set $A_0\subset S^{n-1}$ with $\mu(A_0)=0$ such that for $t\notin A_0$, $K_0(\cdot,t)$ is uniformly continuous on every bounded set of rational numbers. Hence, we can define
$$
K(s,t)=\left\{
\begin{array}{ccl}
\underset{n}{\lim}\, K_0(s_n,t)  &   & \text{if } t\notin A_0, \text{ and } s_n\rightarrow s \text{ with }s_n\in \mathbb Q,  \\
  &   &   \\
0  &   & \text{if } t\in A_0,
\end{array}
\right.
$$
and exactly as in \cite[2.1.4.]{DO1}, it can be checked that $K$ satisfies the strong Carath\'eodory condition (thus, we get (a)), and $K(s,t)=K_0(s,t)$ for every $s\in\mathbb R$ and $\mu$-almost every $t\in S^{n-1}$.

Since for $\lambda\in \mathbb R_+$, it is clear that $\mu_\lambda$ is absolutely continuos with respect to $\mu$, we can take $G_\lambda\in L^1(\mu)$ to be its Radon-Nikodym derivative. We claim that $K(s,t)\leq G_\lambda(t)$ for $s<\lambda$ and $\mu$-almost every $t\in S^{n-1}$. Indeed, let $A\subset S^{n-1}$ be an arbitrary Borel set and $\varepsilon>0$. By definition of $\mu_\lambda$ (see \eqref{def:mu}) we can take an open set $G$ such that $A\subset G$ and
$$
\sup\{\tilde V(f):f\prec G,\,\|f\|_\infty\leq�\lambda\}\leq \mu_\lambda(A)+\varepsilon.
$$
Let $f_n\subset C(S^{n-1})_+$ such that $f_n\prec G$, $\|f_n\|_\infty\leq \lambda$ and $\|f_n- s\chi_A\|_\infty\rightarrow0$. Hence, we have
$$
\nu_s(A)=\tilde V(s\chi_A)=\lim_n \tilde V(f_n)\leq \mu_\lambda(A)+\varepsilon,
$$
and since $\varepsilon>0$ is arbitrary we get $\nu_s(A)\leq \mu_\lambda(A)$ for every $s<\lambda$ and $A\subset S^{n-1}$. Hence, by the Radon-Nikodym theorem it follows that for $s<\lambda$ and $\mu$-almost every $t\in S^{n-1}$
$$
K_0(s,t)\leq G_\lambda(t).
$$
Since $K(s,t)=K_0(s,t)$ for $\mu$-almost every $t\in S^{n-1}$, (b) follows.

Finally, for $s\in\mathbb Q_+$ and $A\in S^{n-1}$ we have
$$
\int_{S^{n-1}} K(s\chi_A(t),t)d\mu(t)=\int_A K_0(s,t)d\mu(t)=\tilde V(s\chi_A).
$$
Hence, equation \eqref{eq:integral} holds for star sets whose radial function is simple and with rational coefficients, that is $\rho_L=\sum_{k=1}^n q_k\chi_{A_k}$. Since these are dense in the star sets with bounded Borel radial function, by continuity of both sides of the equation, the conclusion follows.
\end{proof}

\section{Valuations on $\mathcal S_0^n$ and measures in $\mathbb R^n$}\label{S:measures}
Valuations are often presented as a ``generalization of the notion of measure''. In order to justify their study and to understand their applications, it is important to classify valuations on $\mathcal S_0^n$, distinguishing those which arise from a measure in $\mathbb R^n$ from those which are not.

In this section we find such a classification. Our main tool is the integral representation Theorem \ref{t:integral}, together with the notion of {\em variation} of a valuation, which we define next.

Given a (not necessarily continuous) valuation $V:C(S^{n-1})_+\longrightarrow \mathbb R$, for $f,g\in C(S^{n-1})_+$ with $f\leq g$, we define the {\em variation} of $V$ on the interval $[f,g]$ as
$$
|V|([f,g])=\sup\left\{ \sum_{k=1}^m |V(f_{k})-V(f_{k-1})|\right\},$$
where the supremum is taken over all finite sequences $(f_k)_{k=0}^m$ contained in $C(S^{n-1})_+$ such that $ f=f_0\leq f_1\leq \cdots \leq f_m=g.$

We say that $V$ has bounded variation if, for every $f,g\in C(S^{n-1})_+$ with $f\leq g$, it holds that $|V|([f,g])<\infty$.

It is easy to see that not every continuous valuation has bounded variation: Indeed, consider a function
$$
\theta:\mathbb R_+\longrightarrow \mathbb R_+
$$
such that $\theta(0)=0$ (this condition is not needed, we just impose it for clarity) and such that $\theta$ is continuous but does not have bounded variation (in the classical sense of variation of a function). Let $I=[0,a]$ be an interval where the variation of $\theta$ is not finite. That is
$$
\sup\left\{ \sum_{k=1}^m |\theta(x_{k+1})-\theta(x_k)|: 0\leq x_1\leq \cdots \leq x_m\leq a \right\}=+\infty.
$$
We consider the continous valuation $V:C(S^{n-1})_+\longrightarrow \mathbb R$ defined by
$$
V(f)=\int_{S^{n-1}} \theta(f(t)) dm(t),
$$
where $m$ is the normalized Lebesgue measure in $S^{n-1}$ (see \cite{Vi}). Then, we clearly have
$$
|V|([0,a\uno])\geq \sup\left\{ \sum_{k=1}^m |V(x_{k+1}\uno )-V(x_k\uno)|\right\}=\sup\left\{ \sum_{k=1}^m |\theta(x_{k+1})-\theta(x_k)|\right\},
$$
which is not upper bounded.

\smallskip

Given a valuation $V:C(S^{n-1})_+\longrightarrow \mathbb R$ with bounded variation, we can associate the variation function $|V|:C(S^{n-1})_+\longrightarrow \mathbb R_+$ given by
$$
|V|(f)=|V|([0,f]).
$$
It is clear that $|V|$ is increasing, in the sense that $|V|(f)\leq |V|(g)$ whenever $f\leq g$. We will see next that $|V|$ is actually also a  valuation on $C(S^{n-1})_+$. We need a preliminary lemma first.

\begin{lema}\label{suma intervalos}
Given $f,g,h\in C(S^{n-1})_+$ with $f\leq g \leq h$ we have that
$$
|V|([f,h])=|V|([f,g])+|V|([g,h]).
$$
\end{lema}

\begin{proof}
Let $\epsilon>0$ and take $(f_i)_{i=0}^n, \,(g_j)_{j=0}^m\subset C(S^{n-1})_+$ such that $f=f_0\leq f_1\leq \ldots\leq f_n=g$, $g=g_0\leq g_1\leq \ldots\leq g_m=h$ with
$$
|V|([f,g])\leq \sum_{i=1}^n|V(f_i)-V(f_{i-1})|+\frac\epsilon2,
$$
and
$$
|V|([g,h])\leq \sum_{j=1}^m|V(g_j)-V(g_{j-1})|+\frac\epsilon2.
$$
Considering the yuxtaposition of $(f_i)_{i=0}^n$ and $(g_j)_{j=0}^m$ it follows that
\begin{eqnarray*}
|V|([f,g])+|V|([g,h])&\leq &\sum_{i=1}^n|V(f_i)-V(f_{i-1})|+  \sum_{j=1}^m|V(g_j)-V(g_{j-1})|+\epsilon\\
&\leq& |V|([f,h])+\epsilon,
\end{eqnarray*}
and since $\epsilon>0$ is arbitrary, we get that $|V|([f,g])+|V|([g,h])\leq|V|([f,h])$.

For the converse inequality, we just need to observe that, for every finite sequence $(h_j)_{j=0}^m\subset C(S^{n-1})_+$ with  $f=h_0\leq h_1\leq \cdots \leq h_m=h$, one has
\begin{eqnarray*}
\sum_{i=1}^m |V(h_i)-V(h_{i-1})|&=&\sum_{i=1}^m |V(h_i)+V(g)-V(g)-V(h_{i-1})| \\
&= &\sum_{i=1}^m |V(h_i\vee g)+V(h_i\wedge g)-V(h_{i-1}\vee g)-V(h_{i-1}\wedge g)|\\
&\leq & \sum_{i=1}^m|V(h_i\wedge g)-V(h_{i-1}\wedge g)|+\sum_{i=1}^m|V(h_i\vee g)-V(h_{i-1}\vee g)|\\
&\leq& |V|([f,g])+|V|([g,h]),
\end{eqnarray*}
where the last inequality follows from the fact that
$f=h_0\wedge g \leq h_1\wedge g \leq \cdots \leq h_{m}\wedge g=g$ and $g=h_0\vee g \leq h_1\vee g \leq \cdots \leq h_{m}\vee g=h$.
\end{proof}

\begin{prop}\label{variation is valuation}
Let $V:C(S^{n-1})_+\longrightarrow \mathbb R$ be a  valuation with bounded variation $|V|$. Then $|V|: C(S^{n-1})_+\longrightarrow \mathbb R$ defined by
$$
|V|(f)=|V|([0,f]).
$$
is also a valuation.
\end{prop}

\begin{proof}
Let $V$ be as in the hypothesis and $f,g\in C(S^{n-1})_+$. We choose finite sequences in $C(S^{n-1})_+$, with $0= f_0\leq f_1\leq \cdots \leq f_m=f$, $0=g_0\leq g_1\leq \cdots \leq g_l=g$.

Since $V$ is a valuation, we get

\begin{eqnarray*}
\sum_{j=1}^l \left|V(g_{j})-V(g_{j-1})\right| &=&\sum_{j=1}^l \left|V(g_{j})+V(f)-(V(g_{j-1})+V(f))\right| \\
&=&\sum_{j=1}^l \left|  V(g_{j}\vee f) + V(g_{j}\wedge f) -  V(g_{j-1}\vee f) - V(g_{j-1}\wedge f) \right|\\
&\leq&\sum_{j=1}^l \left|  V(g_{j}\vee f) -  V(g_{j-1}\vee f)\right| + \sum_{j=1}^l \left| V(g_{j}\wedge f) - V(g_{j-1}\wedge f) \right|
\end{eqnarray*}

Therefore, using the fact that $0= f_0\leq f_1\leq \cdots \leq f_m=f\leq f\vee g_1\leq \cdots \leq f\vee g_l=f\vee g$ and that $0\leq g_0\wedge f\leq \cdots \leq g_l\wedge f= g\wedge f$, we obtain that
$$
|V|(f)+|V|(g)\leq |V|(f\vee g)+ |V|(f\wedge g).
$$

For the converse inequality, let $0= \varphi_0\leq \varphi_1\leq \cdots \leq \varphi_m=f\vee g$, $0=\psi_0\leq\psi _1\leq \cdots \leq\psi_l=f\wedge g$ be  finite sequences in $C(S^{n-1})_+$. Note first that for each $1\leq j\leq m$, since $V$ is a valuation, we have
\begin{eqnarray*}
V(\varphi_{j})-V(\varphi_{j-1}) &=&V(\varphi_{j})+V(f)-V(\varphi_{j-1})-V(f) \\
&=&  V(\varphi_{j}\wedge f)  - V(\varphi_{j-1}\wedge f)+ V(\varphi_{j}\vee f) -  V(\varphi_{j-1}\vee f) \\
&=& V(\varphi_{j}\wedge f)  - V(\varphi_{j-1}\wedge f) + V(\varphi_{j}\vee f) +V(g)-  V(\varphi_{j-1}\vee f) -V(g)\\
&=& V(\varphi_{j}\wedge f)  - V(\varphi_{j-1}\wedge f)+ V((\varphi_{j}\vee f)\vee g)) -  V((\varphi_{j-1}\vee f)\vee g) \\
&&  +  V((\varphi_{j}\vee f)\wedge g)) - V((\varphi_{j-1}\vee f)\wedge g) \\
&=&V(\varphi_{j}\wedge f)  - V(\varphi_{j-1}\wedge f)+V((\varphi_{j}\vee f)\wedge g)) - V((\varphi_{j-1}\vee f)\wedge g)
\end{eqnarray*}
Hence, we have
\begin{eqnarray*}
\sum_{j=1}^m |V(\varphi_{j})-V(\varphi_{j-1})|+\sum_{i=1}^l |V(\psi_i)-V(\psi_{i-1})|\leq \sum_{j=1}^m |V(\varphi_{j}\wedge f)-V(\varphi_{j-1}\wedge f)|+\\
+ \sum_{i=1}^l |V(\psi_i)-V(\psi_{i-1})|+ \sum_{j=1}^m |V((\varphi_{j}\vee f)\wedge g)-V((\varphi_{j-1}\vee f)\wedge g)|.
\end{eqnarray*}
Since $0=\psi_0\leq\psi _1\leq \cdots \leq\psi_l=(\varphi_0\vee f)\wedge g\leq\cdots\leq (\varphi_m\wedge f)\vee g=g$ and $0= \varphi_0\wedge f \leq \varphi_1\wedge f\leq \cdots \leq \varphi_m\wedge f=f$, it follows that
$$
|V|(f\vee g)+|V|(f\wedge g)\leq |V|(f)+|V|( g).
$$
\end{proof}

In the next result we show that $|V|$ inherits the continuity of $V$.

\begin{lema}\label{variation continuous}
If $V$ is continuous and has bounded variation, then $|V|$ is also continuous.
\end{lema}

\begin{proof}
Let $f\in C(S^{n-1})_+$ and $\epsilon>0$. Let $(f_i)_{i=0}^m\subset C(S^{n-1})_+$ with $ 0=f_0\leq f_1\leq \cdots \leq f_m=f$ such that
$$
|V|(f)\leq \sum_{i=1}^m |V(f_i)-V(f_{i-1})|+\frac{\epsilon}{4},
$$
and let $(g_j)_{j=0}^n\subset C(S^{n-1})_+$ with  $f=g_0\leq g_1\leq \cdots \leq g_m=f+1$ such that
$$
|V|([f,f+1])\leq \sum_{j=1}^n |V(g_j)-V(g_{j-1})|+\frac{\epsilon}{4}.
$$
By Theorem \ref{t:unif}, there exists $0<\delta<1$ such that whenever $u,v\in[0,f+1]$ with $\|u-v\|_\infty<\delta$ then
$$
|V(u)-V(v)|<\frac{\epsilon}{4\max\{m,n\}}.
$$

Suppose first that $h\in C(S^{n-1})_+$ with $h\leq f$ and $\|f-h\|_\infty<\delta$. Note that for $1\leq i\leq m$ we have $\|f_i\vee h- f_{i-1}\vee h\|_\infty<\delta$, so $|V(f_i\vee h)-V(f_{i-1}\vee h)|<\varepsilon/4m$. Hence, it follows that
\begin{eqnarray*}
|V|(f)&\leq & \sum_{i=1}^m |V(f_i)-V(f_{i-1})|+\frac{\epsilon}{4}\\
&=& \sum_{i=1}^m |V(f_i)+V(h)-V(h)-V(f_{i-1})|+\frac{\epsilon}{4}\\
&=& \sum_{i=1}^m |V(f_i\vee h)+V(f_i\wedge h)-V(f_{i-1}\vee h)-V(f_{i-1}\wedge h)|+\frac{\epsilon}{4}\\
&\leq&\sum_{i=1}^m |V(f_i\wedge h)-V(f_{i-1}\wedge h)|+\sum_{i=1}^m |V(f_i\vee h)-V(f_{i-1}\vee h)|+\frac{\epsilon}{4}\\
&\leq&|V|(h)+\frac\epsilon2.
\end{eqnarray*}
Since $|V|(h)\leq |V|(f)$, we get that
$$
| |V|(f)-|V|(h)|\leq \frac\epsilon2
$$
whenever $\|f-h\|_\infty<\delta$ and $h\leq f$.

Now, suppose that $h\in C(S^{n-1})_+$ with $\|f-h\|_\infty<\delta$ and $f\leq h$. Note that for $1\leq j\leq n$ we have $f\leq g_j\wedge h\leq h$, so in particular $\|g_j\wedge h-g_{j-1}\wedge h\|_\infty<\delta$, and so $|V(g_j\wedge h)-V(g_{j-1}\wedge h)|<\epsilon/4n$. Moreover, by Lemma \ref{suma intervalos} it follows that
\begin{eqnarray*}
|V|(h)-|V|(f)&= & |V|([f,h])=|V|([f,f+1])-|V|([h,f+1])\\
&\leq& \sum_{j=1}^n |V(g_j)-V(g_{j-1})|+\frac{\epsilon}{4}-V([h,f+1])\\
&=& \sum_{j=1}^n |V(g_j\vee h)+V(g_j\wedge h)-V(g_{j-1}\vee h)-V(g_{j-1}\wedge h)|\\
&&+\frac{\epsilon}{4}-V([h,f+1])\\
&\leq& \sum_{j=1}^n |V(g_j\wedge h)-V(g_{j-1}\wedge h)|+\frac{\epsilon}{4}\\
&&+\sum_{j=1}^n|V(g_j\vee h)-V(g_{j-1}\vee h)|-V([h,f+1])\\
&<&\sum_{j=1}^n |V(g_j\wedge h)-V(g_{j-1}\wedge h)|+\frac{\epsilon}{4}\leq\frac\epsilon2.
\end{eqnarray*}
And since $|V|(h)\geq |V|(f)$, we also get that
$$
| |V|(f)-|V|(h)|\leq \frac\epsilon2
$$
whenever $\|f-h\|_\infty<\delta$ and  $f\leq h$.

Finally, for arbitrary $h\in C(S^{n-1})_+$ with $\|f-h\|_\infty<\delta$, by Proposition \ref{variation is valuation} we have
$$
|V|(f)-|V|(h)=|V|(f)-|V|(f\vee h)+|V|(f)-|V|(f\wedge h).
$$
Since $\|f-f\vee h\|_\infty<\delta$ and $\|f-f\wedge h\|_\infty<\delta$, by the above we get that
$$
||V|(f)-|V|(h)|\leq \epsilon.
$$
\end{proof}

Finally, we can prove our classification result Theorem \ref{t:variacion}.

\begin{proof}[Proof of Theorem \ref{t:variacion}]
$(1)\Rightarrow (2)$: Suppose there exists a (signed) countably additive measure $\nu$ on the Borel sets of $\mathbb R^n$ such that, for every $L\in \mathcal S_0^n$, $\nu(L)=V(\rho_L).$ Let $f\leq g\in C(S^{n-1})_+$. For $(f_i)_{i=0}^m\subset C(S^{n-1})_+$ such that $f=f_0\leq f_1\leq \cdots\leq f_m=g$ take $L_{f_i}\in\mathcal S_0^n$ with $\rho_{L_{f_i}}=f_i$. Let us consider the Jordan decomposition of the measure $nu$ as $\nu=\nu_+-\nu_-$ (cf. \cite[\S 29Theorem B]{Halmos}). For $1\leq i\leq m$ we have
\begin{eqnarray*}
|V(f_i)-V(f_{i-1})|&=&|\nu(L_{f_i})-\nu(L_{f_{i-1}})|\\
&=&|\nu_+(L_{f_i})-\nu_-(L_{f_i})-\nu_+(L_{f_{i-1}})-\nu_-(L_{f_{i-1}})|\\
&\leq&|\nu_+(L_{f_i})-\nu_+(L_{f_{i-1}})|+|\nu_-(L_{f_i})-\nu_-(L_{f_{i-1}})|\\
&=&\nu_+(L_{f_i})-\nu_+(L_{f_{i-1}})+\nu_-(L_{f_i})-\nu_-(L_{f_{i-1}}).
\end{eqnarray*}
Therefore, we get
$$
\sum_{i=1}^m|V(f_i)-V(f_{i-1})|\leq\nu_+(L_{g})-\nu_+(L_f)+\nu_-(L_g)-\nu_-(L_f),
$$
which yields that $|V|([f,g])<\infty$ as claimed.

\medskip

$(2)\Rightarrow(3)$: If $V$ has bounded variation, then we can write
$$
V=|V|-(|V|-V).
$$
By Proposition \ref{variation is valuation} and Lemma \ref{variation continuous} we have that $|V|$ is an increasing continuous valuation. Hence, it is enough to show that $|V|-V$ is also increasing. To this end, pick $f\leq g$ in $C(S^{n-1})_+$ and note that
$$
V(g)-V(f)\leq |V(g)-V(f)|\leq|V|([f,g])=|V|(g)-|V|(f).
$$
Therefore, it holds that $|V|(f)-V(f)\leq |V|(g)-V(g)$.

\medskip

$(3)\Rightarrow(4)$: Clearly, it is enough to show that if $V$ is monotone increasing, then for $\mu$-almost every $t\in S^{n-1}$, $K(\cdot, t)$ is increasing. This actually follows from the construction of $K$ given in the proof of Theorem \ref{t:integral}. Indeed, recall that we can define for every $s\in \mathbb R_+$ the function $K_0(s,\cdot)$ as the Radon-Nikodym derivative with respect to $\mu$ of the measure given by $\nu_s(A)=V(s\chi_A)$ for $A\in\Sigma_n$. Moreover, it is seen in the proof of Theorem \ref{t:integral} (see also the Appendix for more details) that $K_0(s,t)=K(s,t)$ for every $s\in\mathbb R_+$ and $\mu$-almost every $t\in S^{n-1}$. Now, if $V$ is monotone increasing, and $s_1,s_2\in \mathbb R_+$ are such that $s_1\leq s_2$, then for every $A\in \Sigma_n$ we have
$$
\nu_{s_1}(A)\leq \nu_{s_2}(A)
$$
which yields that for $\mu$-almost every $t\in S^{n-1}$
$$
K(s_1,t)=K_0(s_1,t)\leq K_0(s_2,t)=K(s_2,t).
$$
Thus, $K(\cdot,t)$ is increasing for $\mu$-almost every $t\in S^{n-1}$.

\medskip

$(4)\Rightarrow (1)$: We will see that if $K(\cdot,t)$ is continuous increasing for $\mu$-almost every $t\in S^{n-1}$, then there is a (positive) countably additive measure $\nu$ with $\nu(L)=V(\rho_L)$ for every star body $L\in \mathcal S_0^n$. Since every continuous function of bounded variation can be written as the difference of continuous increasing functions, the conclusion will follow.

Let us consider the semiring of subsets of $\mathbb R_+\times S^{n-1}$ given by
$$
\mathcal D=\{[a,b)\times A: a,b\in \mathbb R_+,\,a<b\,\textrm{and} \,A\in \Sigma_n\}
$$
and define $\nu:\mathcal D\rightarrow \mathbb R_+$ by
$$
\nu([a,b)\times A)=\int_A K(b,t)-K(a,t)d\mu(t).
$$

A standard argument shows that $\nu$ can be extended to a Borel measure on $\mathbb R^n$ (see Lemma \ref{l:extendnu} for details).

Finally, note that for every simple Borel star set $L\subset \mathbb R^n$ we have $\nu(L)=V(\rho_L)$. Indeed, let $(a_i)_{i=1}^m\subset\mathbb R_+$ and pairwise disjoint $(A_i)_{i=1}^m\subset\Sigma_n$ such that $\rho_L=\sum_{i=1}^m a_i\chi_{A_i}$. It follows that
\begin{eqnarray*}
\nu(L)&=&\nu\left(\bigcup_{i=1}^m [0,a_i)\times A_i\right)=\sum_{i=1}^m\nu\left([0,a_i)\times A_i\right)\\
&=&\sum_{i=1}^m\int_{A_i} K(a_i,t)d\mu(t)=\int_{S^{n-1}}K(\sum_{i=1}^m a_i\chi_{A_i}(t),t)d\mu(t)=V(\rho_L).
\end{eqnarray*}
Now, let $L\subset \mathbb R^n$ be a star body (with continuous radial function), and take an increasing sequence $(L_k)_{k\in\mathbb N}$ of simple Borel star sets such that $L=\bigcup_{k\in\mathbb N} L_k$. Therefore, it follows that
$$
\nu(L)=\lim_k \nu(L_k)=\lim_k V(\rho_{L_k})=V(\rho_L).
$$
\end{proof}

\begin{proof}[Proof of Corollary \ref{c:variacion}]
The proof follows exactly the same lines as the previous proof, with big simplifications due to rotational invariance. 

The implications (1) implies (2) and (2) implies (3) are exactly as in the previous proof, just noting the easy fact that the variation of a rotationally invariant valuation is also rotationally invariant. To show (3) implies (4), we use  \cite[Corollary 4.1]{TraVi}, and the definition of $\theta$ thereof, and we obtain two continuous monotonic increasing functions $\theta_1, \theta_2$ representing each of the monotonic increasing valuations in (3). Then $\theta=\theta_1-\theta_2$. 

Finally, to see that (4) implies (1), we just need to define $\nu$ on the intervals $[a,b]$ by $$\nu([a,b])=V(b\uno)-V(a\uno)$$
and check that is allows us to define, in a simpler way as the previous proof, a measure $\nu$ verifying (1). 
\end{proof}

\appendix

\section{}

\begin{lema}\label{l:unifrational}
Suppose $K_0:\mathbb R_+\times S^{n-1}\rightarrow \mathbb R$ is given by
$$
K_0(s,t)=\Phi(s\chi_{S^{n-1}})(t),
$$
where $\Phi:B(S^{n-1})_+\rightarrow L_1(\mu)$ is the mapping given in Proposition \ref{p:fi}, then there is a set $A_0\subset S^{n-1}$ with $\mu(A_0)=0$ such that for every $t\notin A_0$,
$K_0(\cdot,t)$ is uniformly continuous on every bounded set of rational numbers.
\end{lema}

We will follow the same approach as in \cite[2.1.3.]{DO1} or \cite[Lemma 11]{CF}. Before the proof, recall that for each $s\in \mathbb R_+$, $\Phi(s\chi_{S^{n-1}})\in L_1(\mu)$ is the Radon-Nikodym derivative of the measure $\nu_s$ with respect to $\mu$, where $\nu_s(A)=\tilde V(s\chi_A)$ for $A\in \Sigma_n$.

For $\delta>0$, $\lambda>0$, and $A\in \Sigma_n$, let
$$
\omega_{\lambda}(\delta,A)=\sup\Big\{\int_A|K_0(s,t)-K_0(s',t)|d\mu(t):\,s,s'\in[0,\lambda],\,|s-s'|\leq\delta\Big\},
$$
and let
$$
\omega_\lambda(\delta)=\sup\Big\{\sum_{i=1}^m\omega_\lambda(\delta,A_i):\,\bigcup_{i=1}^m A_i=S^{n-1},\, A_i\cap A_j=\emptyset,\textrm{ for }i\neq j\Big\}.
$$

\begin{lema}\label{l:omega}
For every $\lambda>0$, we have that
$$
\lim_{\delta\rightarrow 0} \omega_\lambda(\delta)=0.
$$
\end{lema}

\begin{proof}
Given $\varepsilon>0$, by Theorem \ref{t:unif}, we know that $\tilde V$ is uniformly continuous on bounded sets, so there is $\delta>0$ such that $|\tilde V(f)-\tilde V(g)|\leq \frac{\varepsilon}{3}$ whenever $f,g\in B(S^{n-1})$ satisfy $\|f\|_\infty,\|g\|_\infty\leq \lambda$ and $\|f-g\|_\infty<\delta$.

Let $(A_i)_{i=1}^m\subset \Sigma_n$ pairwise disjoint with $\bigcup_{i=1}^m A_i=S^{n-1}$ such that
$$
\omega_\lambda(\delta)\leq \sum_{i=1}^m\omega_\lambda(\delta,A_i)+\frac{\varepsilon}{3}
$$
For $1\leq i\leq m$, let $s_i, s'_i\in[0,\lambda]$ with $|s_i-s'_i|\leq\delta$, such that
$$
\omega_{\lambda}(\delta, A_i)\leq \int_{A_i}|K_0(s_i,t)-K_0(s'_i,t)|d\mu(t)+\frac{\varepsilon}{3m}.
$$
Let $A_i^+=\{t\in A_i: K_0(s_i,t)\geq K_0(s'_i,t)\}$ and $A_i^-=\{t\in A_i: K_0(s_i,t)< K_0(s'_i,t)\}$, which belong clearly to $\Sigma_n$.

Now, if we set
$$
f=\sum_{i=1}^m s_i \chi_{A_i^+}+s'_i\chi_{A_i^-},\quad\textrm{and }\quad g=\sum_{i=1}^m s'_i \chi_{A_i^+}+s_i\chi_{A_i^-},
$$
then we clearly have $\|f\|_\infty,\|g\|_\infty\leq \lambda$ and $\|f-g\|_\infty\leq\delta$. Hence, $|\tilde V(f)-\tilde V(g)|\leq \frac{\varepsilon}{3}$, which yields
\begin{eqnarray*}
\omega_\lambda(\delta)&\leq &\sum_{i=1}^m  \int_{A_i}|K_0(s_i,t)-K_0(s'_i,t)|d\mu(t)+\frac{2\varepsilon}{3}\\
&=&\sum_{i=1}^m  \int_{A_i^+} K_0(s_i,t)-K_0(s'_i,t)d\mu(t)+  \int_{A_i^-}K_0(s'_i,t)-K_0(s_i,t)d\mu(t)+\frac{2\varepsilon}{3}\\
&=&\sum_{i=1}^m \tilde V(s_i\chi_{A_i^+}) -\tilde V(s'_i\chi_{A_i^+})+\tilde V(s'_i\chi_{A_i^-})-\tilde V(s_i\chi_{A_i^-})+\frac{2\varepsilon}{3}\\
&=&\tilde V(f)-\tilde V(g)+\frac{2\varepsilon}{3}\leq \varepsilon.
\end{eqnarray*}
\end{proof}

\begin{proof}[Proof of Lemma \ref{l:unifrational}]
For $k\in \mathbb N$, let $S_k=\mathbb Q\cap[0,k]$. Given $\delta>0$, and $\varepsilon>0$, set
$$
A(\delta,\varepsilon)=\Big\{t\in S^{n-1}:\sup\{|K_0(s,t)-K_0(s',t)|:s,s'\in S_k,\,|s-s'|\leq \delta\}>\varepsilon.
$$
Also, given $s,s'\in S_k$ with $|s-s'|\leq \delta$, set
$$
B(s,s',\delta,\varepsilon)=\{t\in S^{n-1}:|K_0(s,t)-K_0(s',t)|>\varepsilon\}.
$$
Let $(s_i,s'_i)_{i\in\mathbb N}$ be an enumeration of all pairs $(s,s')$ where $s,s'\in S_k$ and $|s-s'|\leq \delta$. Let $A_1(\delta,\varepsilon)=B(s_1,s'_1,\delta,\varepsilon)$ and
$$
A_i(\delta,\varepsilon)=B(s_i,s'_i,\delta,\varepsilon)\backslash\bigcup_{j=1}^{i-1} A_j(\delta,\varepsilon).
$$
In this way, we obtain a sequence of pairwise disjoint sets such that $\bigcup_{i=1}^\infty A_i(\delta,\varepsilon)=A(\delta,\varepsilon)$.

Now, it follows that
$$
\varepsilon \mu\Big(A(\delta,\varepsilon)\Big)\leq\sum_{i=1}^\infty\int_{A_i(\delta,\varepsilon)}|K_0(s_i,t)-K_0(s'_i,t)|d\mu(t)\leq\omega_k(\delta).
$$
Therefore, by Lemma \ref{l:omega}, for every $\varepsilon>0$ we get that
$$
\lim_{\delta\rightarrow 0} \mu\Big(A(\delta,\varepsilon)\Big)=0.
$$

Now, for each $\varepsilon>0$, pick a sequence $\delta_m\rightarrow 0$ such that
$$
\sum_{m=1}^\infty \mu\Big(A(\delta_m,\varepsilon)\Big)<\infty,
$$
and set
$$
A(\varepsilon)=\bigcap_{k=1}^\infty\bigcup_{m=k}^\infty A(\delta_m,\varepsilon).
$$
It is clear that $\mu\Big(A(\varepsilon)\Big)=0$. Now, take $\varepsilon_j\rightarrow 0$ and set $A^k=\bigcup_{j=1}^\infty A(\varepsilon_j)$, which also satisfies $\mu(A^k)=0$. It is easy to check that for every $t\in S^{n-1}\backslash A^k$, $K_0(\cdot, t)$ is uniformly continuous on $S_k$.

Finally, set $A_0=\bigcup_{k=1}^\infty A^k$, which is the required set with $\mu(A_0)=0$ and such that for every $t\notin A_0$
$K_0(\cdot,t)$ is uniformly continuous on every bounded set of rational numbers.

\end{proof}

\begin{lema}\label{l:extendnu}
Let $K:\mathbb R_+\times S^{n-1}\rightarrow \mathbb R$ such that $K(s,\cdot)$ is measurable for every $s\in\mathbb R_+$ and $K(\cdot,t)$ is continuous increasing for $\mu$-almost every $t\in S^{n-1}$. Let $\mathcal D$ be the semiring of subsets of $\mathbb R_+\times S^{n-1}$ given by
$$
\mathcal D=\{[a,b)\times A: a,b\in \mathbb R_+,\,a<b\,\textrm{and} \,A\in \Sigma_n\}
$$
and define $\nu:\mathcal D\rightarrow \mathbb R_+$ by
$$
\nu([a,b)\times A)=\int_A K(b,t)-K(a,t)d\mu(t).
$$
Then $\nu$ can be extended to a Borel measure on $\mathbb R^n$.
\end{lema}

\begin{proof}
We follow a similar approach as in the construction of Lebesgue measure.

Clearly, $\nu$ is finitely additive on $\mathcal D$, in the sense that for any disjoint family $([a_i,b_i)\times A_i)_{i=1}^m$ such that $\bigcup_{i=1}^m [a_i,b_i)\times A_i=[a,b)\times A$ with $[a,b)\times A\in \mathcal D$ it follows that
$$
\nu([a,b)\times A)=\sum_{i=1}^m \nu([a_i,b_i)\times A_i).
$$

We claim that $\nu$ is actually countably additive on $\mathcal D$. Indeed, suppose $[a,b)\times A=\bigcup_{i=1}^\infty [a_i,b_i)\times A_i$. First, note that for every $m\in \mathbb N$ there exist pairwise disjoint $([c_j,d_j)\times B_j)_{j=1}^{N}$ in $\mathcal D$ such that
$$
[a,b)\times A= \bigcup_{i=1}^m [a_i,b_i)\times A_i \cup \bigcup_{j=1}^N [c_j,d_j)\times B_j.
$$
Therefore, since $\nu:\mathcal D\rightarrow \mathbb R_+$ is finitely additive it follows that $\nu([a,b)\times A)\geq \sum_{i=1}^m \nu([a_i,b_i)\times A_i)$, which, as $m\in \mathbb N$ is arbitrary, implies that
$$
\nu([a,b)\times A)\geq \sum_{i=1}^\infty \nu([a_i,b_i)\times A_i).
$$
For the converse inequality, let $\epsilon>0$. Using the regularity of $\mu$ we can find a compact set $K_A\subset A$ such that
\begin{equation}\label{eq:regular}
\int_{A\backslash K_A} K(b,t)d\mu(t) <\epsilon,
\end{equation}
as well as open sets $(U_n)_{n=1}^\infty\subset \Sigma_n$ such that for every $n\in \mathbb N$ we have $A_n\subset U_n$ and
\begin{equation}\label{Un-An}
\int_{U_n\backslash A_n} K(b_n,t)-K(a_n,t)d\mu(t) <\frac\epsilon{2^n},
\end{equation}
Since $K(\cdot,t)$ is continuous for $\mu$-almost every $t\in S^{n-1}$, we can find $\delta>0$ such that
\begin{equation}\label{eq:contdelta}
\int_A K(b,t)-K(b-\delta,t)d\mu(t)<\epsilon,
\end{equation}
and for each $n\in\mathbb N$ we can also take $\delta_n>0$ such that
\begin{equation}\label{Unandelta}
\int_{U_n} K(a_n,t)-K(a_n-\delta_n,t)d\mu(t)<\frac\epsilon{2^n}.
\end{equation}
Since $[a,b-\delta]\times K_A\subset \bigcup_{n=1}^\infty (a_n-\delta_n,b_n)\times U_n$, by compactness there exists a finite set $F\subset \mathbb N$ such that
$$
[a,b-\delta]\times K_A\subset \bigcup_{n\in F} (a_n-\delta_n,b_n)\times U_n.
$$
In particular, we have
$$
[a,b-\delta)\times K_A\subset \bigcup_{n\in F} [a_n-\delta_n,b_n)\times U_n,
$$
so that, using \eqref{Unandelta} and \eqref{Un-An}, it follows that
\begin{eqnarray}\label{eq:nusum}
\nu([a,b-\delta)\times K_A)&\leq& \sum_{n\in F} \nu([a_n-\delta_n,b_n)\times U_n)\\
\nonumber&=&\sum_{n\in F} \int_{U_n} K(b_n,t)-K(a_n-\delta_n,t)d\mu(t)\\
\nonumber&<&\sum_{n\in F} \int_{U_n} K(b_n,t)-K(a_n,t)d\mu(t)+\epsilon\\
\nonumber&<&\sum_{n\in F} \int_{A_n} K(b_n,t)-K(a_n,t)d\mu(t)+2\epsilon.
\end{eqnarray}

Now, using the monotonicity of $K(\cdot,t)$ and \eqref{eq:regular}, we have
\begin{equation}\label{eq:AmenosK_A}
\nu\left([a,b-\delta)\times (A\backslash K_A)\right)=\int_{A\backslash K_A} K(b-\delta,t)-K(a,t)d\mu(t)\leq \int_{A\backslash K_A} K(b,t)d\mu(t)<\epsilon.
\end{equation}

Moreover, \eqref{eq:contdelta} yields
\begin{equation}\label{eq:deltaA}
\nu([b-\delta,b)\times A)\leq \int_A K(b,t)-K(b-\delta,t)d\mu(t)<\epsilon.
\end{equation}

Therefore, putting together \eqref{eq:nusum}, \eqref{eq:AmenosK_A} and \eqref{eq:deltaA}, we get
\begin{eqnarray*}
\nu\left([a,b)\times A\right)&=&\nu\left([a,b-\delta)\times K_A\right)+\nu\left([a,b-\delta)\times (A\backslash K_A)\right)+\nu\left([b-\delta,b)\times A\right)\\
&\leq& \sum_{n\in F}\nu\left([a_n,b_n)\times A_n\right)+4\epsilon\leq \sum_{n=1}^\infty\nu\left([a_n,b_n)\times A_n\right)+4\epsilon.
\end{eqnarray*}
Since $\epsilon>0$ was arbitrary, it follows that $\nu:\mathcal D\rightarrow \mathbb R_+$ is countably additive as claimed.

Noting that $\mathbb R^n\backslash\{0\}$ is homeomorphic to $(0,\infty)\times S^{n-1}$, it is easy to check that $\mathcal D$ generates the $\sigma$-algebra of all Borel subsets of $\mathbb R^n$. Hence, a standard argument (cf. \cite[Proposition 3.2.4]{Dudley}) yields that $\nu$ can be extended to a Borel measure on $\mathbb R^n$.
\end{proof}

\end{document}